\documentclass{amsart}
\usepackage{geometry}

\usepackage{graphicx}
\usepackage{float} 
\usepackage{color}
\usepackage{amssymb}
\usepackage{amsmath}
\usepackage{yhmath}
\usepackage{mathrsfs}
\usepackage[mathscr]{eucal}
\usepackage[all]{xy}
\usepackage{amsthm}
\usepackage{titletoc}
\usepackage{tikz}
\usepackage[colorlinks,linkcolor=blue,anchorcolor=blue,citecolor=blue]{hyperref}

\theoremstyle{plain}
\newtheorem{theorem}{Theorem}[section]
\newtheorem{lemma}[theorem]{Lemma}
\newtheorem{proposition}[theorem]{Proposition}
\newtheorem{corollary}[theorem]{Corollary}
\theoremstyle{definition}
\newtheorem{definition}{Definition}

\newtheorem{remark}{Remark}
\numberwithin{equation}{section}

\newcommand{\ML}{\mathcal{ML}_g}

\newcommand{\Th}{\operatorname{Th}}
\newcommand{\Mod}{\operatorname{Mod}_g}
\newcommand{\Mir}{B_{\operatorname{Mir}}}

\newcommand{\Pg}{\mathcal{P}_g}

\usepackage{microtype}

\begin{document}

\title{Volume of unit balls associated to quadratic differentials}

\author{Weixu Su and Shenxing Zhang }

\address{Weixu Su: School of Mathematics, Sun Yat-sen University, Guangzhou 510275, China}
\email{suwx9@mail.sysu.edu.cn}

\address{Shenxing Zhang: School of Mathematics and Statistics, Nanjing University of Science and Technology, Nanjing, 210094, China} \email{21110180030@m.fudan.edu.cn}

\date{\today}

\maketitle


\noindent

\begin{abstract}
Associated to a holomorphic quadratic differential is a unit ball of the measured lamination space. The Thurston volume of the unit ball defines a function on the moduli space. We show that the volume function is not proper and characterize when it tends to infinity. We prove that the volume function is $p$-integrable for any $0<p<1$. 

 \medskip


\medskip

\end{abstract}

\section{Introduction}

Holomorphic quadratic differentials on Riemann surfaces play an important role in 
Teichm\"uller theory. In this paper, we focus on geodesic lengths of simple closed curves on the flat surfaces determined by holomorphic quadratic differentials.

 
Let $X$ be a Riemann surface of genus $g\geq 2$.
A \emph{holomorphic quadratic differential} $q$ is a $(2,0)$ form on $X$
locally given by 
$$q=q(z) dz^2,$$
where $q(z)$ is a holomorphic function.

Let $(X,q)$ be a non-zero holomorphic quadratic differential. Denote the set of zeros of $q$ by 
$\Sigma(q)$. Then $|q|^{1/2}$ defines a flat metric on $X$ with singularities at
$\Sigma(q)$.  A zero of order
 $k$ corresponds to a singularity of angle $(k+2)\pi$. 
Any essential simple closed curve $\gamma$ on $X$ is freely homotopic to a closed geodesic with respect to the metric $|q|^{1/2}$.
Typically a geodesic representation in a homotopy class of $\gamma$ 
is realized by a union of \emph{saddle connections}, i.e., geodesic segments going from one zero of $q$
to the other. The geodesic representation is unique except that there is a maximal flat cylinder
swept out by closed geodesics homotopic to $\gamma$.

Denote the length of a geodesic representation of $\gamma$ by $\ell_q(\gamma)$. 
The length function $\ell_q(\gamma)$ extends continuously to the space of measured foliations, or equivalently, the space of measured geodesic laminations.
In fact, Duchin-Leininger-Rafi \cite[Proposition 26]{DLR2010}\ showed that there is a natural geodesic current $L_q$ associated to $q$ such that the intersection number of $L_q$ with any simple closed curve $\gamma$ is exactly 
$\ell_q(\gamma)$. 

Let $\ML$ be the space of measured geodesic laminations on $X$.  
There is a natural volume form $\mu_{\Th}$ on $\ML$ called the 
\emph{Thurston measure}, 
which is induced by the  piecewise linear integral structure of $\ML$.
The Thurston measure $\mu_{\Th}$ is invariant and ergodic under the action of the mapping class group. 
For more details we refer the reader to  \cite{ES}.

\begin{definition}[Volume function] Let $(X,q)$ be a non-zero holomorphic quadratic differential. 
We define the unit ball associated to $(X,q)$ by
$$U(X,q) = \{\lambda \in \ML \ | \ \ell_q(\lambda) \leq 1\}.$$
The \emph{Thurston volume} of  $U(X,q)$ is denoted by 
$B(X,q)$.
\end{definition}

As shown by \cite[Theorem 1]{DLR2010}, the lengths of simple closed
curves determine a quadratic differential up to rotation,
this means that
the shape of the unit ball $U(X,q)$ uniquely determines the flat metric induced by $q$. 

The function $B(X,q)$ is an analogue of the \emph{Mirzakhani function}. 
Recall that for any hyperbolic surface $X$, the hyperbolic length functions 
of measured geodesic laminations induce a  unit ball in $\ML$.
The Thurston measure of such a unit ball is now called the Mirzakhani function and denoted by 
$\Mir(X)$.
Mirzakhani \cite[Theorem 3.3]{Mir2008}
shows that $\Mir(X)$ is proper on the moduli space of Riemann surfaces $\mathcal{M}_g$ and integrable with respect to the Weil-Petersson volume form.
Recently,
Arana-Herrera and Athreya \cite{AA}
improve Mirzakhani's result by showing that
$\Mir(X)$ is square-integrable.

A remarkable result due to Mirzakhani is that
for any given simple closed curve $\gamma_0$ on $X$ 
\begin{equation*}\label{equ:Mir}
 \lim_{L \to \infty} \frac{\# \{\gamma \in \Mod \cdot \gamma_0 : \ell_X(\gamma) \leq L \}}{L^{6g-6}} = \frac{c(\gamma_0)}{b_g} \Mir(X).
\end{equation*}
For any given simple closed curve $\gamma_0$ on $X$, we also have  
$$\lim_{L \to \infty} \frac{\# \{\gamma \in \Mod \cdot \gamma_0 : \ell_q(\gamma) \leq L \}}{L^{6g-6}} = \frac{c(\gamma_0)}{b_g} B(X,q).$$
This follows from a result of Erlandsson-Souto \cite[Theorem 1.2]{ES},
which holds 
for any filling geodesic current. We believe that $B(X,q)$, 
considered as a function on the moduli space of quadratic differentials, will  be useful to understand random flat surfaces.

\subsection{Main theorems}
Let $\mathcal{M}_g$ be the moduli space of Riemann surfaces of genus $g\geq 2$.
Let $X\in \mathcal{M}_g$, 
endowed with the complete hyperbolic metric.
For a holomorphic quadratic differential $(X,q)$ and a measurable subset $A$ of $X$, we define 
$$m(A, q) = \int_A |q|.$$

Let $Q^1\mathcal{M}_{g}$ be the moduli space of quadratic differentials 
of unit area, i.e., $m(X,q)=1$.
A sequence of $(X_n,q_n)$ in $Q^1\mathcal{M}_{g}$ is called \emph{divergent}
if it leaves any compact subset of $Q^1\mathcal{M}_{g}$.
This is equivalent to requiring that the systole, i.e., the
 length of the shortest simple closed geodesic of $X_n$ or $|q_n|^{1/2}$
tends to $0$ as $n\to\infty$.

Let $(X_\infty, q_\infty)$ be a quadratic differential on a noded Riemann surface,
which is assumed to be the limit of a divergent sequence $ (X_n,q_n)$ in $Q^1\mathcal{M}_{g}$. 
Assume that $\alpha_1, \cdots, \alpha_k$ is the maximal collection of pairwise disjoint simple closed curves
that are pinched to the nodes of $X_\infty$. 
We assume that each $\alpha_j$ is homotopic to a maximal flat cylinder $C_n(\alpha_j)$  on $(X_n,q_n)$ (when there is no such a cylinder,  we consider $C_n(\alpha_j)$ to be degenerate and assume that $m(C_n(\alpha_j), q_n)=0$). 
 For simplicity of notation,
we write  $C_j$ instead of  $C_n(\alpha_j)$ when no confusion can arise. 

Assume that
$$X_n \setminus \bigcup_{j=1}^k C_j = \bigcup_{j=1}^\ell \Omega_j,$$
where each $\Omega_j$ is a connected component of $X_n \setminus \bigcup_{j=1}^k C_j $. A connected component $\Omega_j$ is \emph{stable} if it is not a pair of pants.
The decomposition of $(X_n,q_n)$ into the union of $C_j$ and $\Omega_j$
is called the \emph{thick-thin decomposition}.

We now introduce the following notion of \emph{regular sequence}. 
\begin{definition}
Let $\epsilon>0$ be a sufficiently small constant. 
A sequence of $(X_n,q_n)$ in $Q^1\mathcal{M}_{g}$ is \emph{regular} if 
it satisfies: 
\begin{itemize}
  \item For $n$ sufficiently large there is a maximal collection of pairwise disjoint simple closed curves $\alpha_1, \cdots, \alpha_k$ such that $\ell_{X_n}(\alpha_j) \to 0 $ as $n\to\infty$. 
  \item  In the thick-thin decomposition of $(X_n,q_n)$ (where the cylinders are corresponding to $\alpha_1, \cdots, \alpha_k$), we have 
some $\delta>0$ independent of $n$ such that $m(C_j, q_n)>\delta$ for every cylinder $C_j$ and 
 $m(\Omega_j, q_n)>\delta$ 
 for every stable component $\Omega_j$.
\end{itemize}

\end{definition}

Our first result is to show that the function $B(X,q)$ on $Q^1\mathcal{M}_{g}$
is not proper, which
behaves differently from $\Mir(X)$. 
More precisely, we show that 
 $B(X_n,q_n)$ are bounded if and only if $(X_n, q_n)$ is a regular sequence. 

\begin{theorem}\label{thm:limit}
Let $\left\{(X_n,q_n)\right\}$ be a divergent sequence  in  $Q^1\mathcal{M}_{g}$.
 Then $B(X_n, q_n) \to \infty$  except the case that the sequence has a regular subsequence. 
 When the sequence $(X_n,q_n)\in Q^1\mathcal{M}_{g}$ is regular, $B(X_n, q_n)$  are  bounded above by a constant depending only on the limit $(X_\infty, q_\infty)$.
\end{theorem} 

 In particular, we have  $B(X_n, q_n) \to \infty$ if there is no loss of mass in the limit. 

\begin{corollary}\label{thm:unit}
Let $(X_n,q_n)\in Q^1\mathcal{M}_{g}$ be  a divergent sequence of quadratic differentials  such that $(X_n, q_n) \to (X_\infty, q_\infty)$.
If $m(X_\infty, q_\infty)=1$, then $B(X_n, q_n) \to \infty$.
\end{corollary}

On a hyperbolic surface $X$, when a simple closed geodesic $\gamma$ has small length say $\ell_X(\gamma) = \epsilon$,
there exists a collar neighborhood of $\gamma$ with width comparable to $\log \frac{1}{\epsilon}$. Unlike hyperbolic metrics, there is no standard collar lemma for quadratic differentials. However, there is another version of collar lemma due to Minsky and developed by Rafi.
They show that either there is a maximal flat cylinder containing $\gamma$ or
 an \emph{expanding annulus} containing $\gamma$ as a boundary component, with conformal modulus comparable to
$1/\epsilon$.
Using the collar lemma for quadratic differentials, we prove 
 
\begin{theorem} \label{thm:int} 
Given any $0<p<1$, 
the function $B(X,q)$ is $p$-integrable on $Q^1\mathcal{M}_{g}$ with respect to the Masur-Veech measure. 
\end{theorem}

Unfortunately, we are not able show that $B(X,q)$ is integrable. 

\subsection*{Outline of the paper}
In Section \ref{sec:pre}, we present some background materials on measured laminations and the Dehn-Thurston parametrization for  integral measured laminations. In Section \ref{sec:limit}, we first recall the Deligne-Mumford compactifications of the space of quadratic differentials. We apply the thick-thin decomposition of quadratic differential to give a lower bound of $B(X,q)$
and prove Theorem \ref{thm:limit}. In Section \ref{sec:collar},
we establish an upper bound of $B(X,q)$, which is base on the collar lemma of quadratic differentials due to Minsky and Rafi. 
We prove Theorem \ref{thm:int} using the upper bound estimate. 

\subsection*{Notation} We write $f_2 = O(f_1)$ or $f_2 \preceq f_1$ to mean $f_2 \leq C f_1$, and 
$f_1 \asymp f_2$ to mean $f_1/C \leq f_2 \leq C f_1$, where $C>0$ is some constant that
depends only on the genus of the surface.

\bigskip
\noindent
{\bf Acknowledgements.} We would like to thank Dawei Chen and Huiping Pan for their interests in this work and useful conversations.

\section{Preliminaries}\label{sec:pre}

In this section, we review the notions of measured laminations and train track coordinates. For further background, see \cite{PH,ES}. 

Under the Dehn-Thurston coordinates,
 integral multi-curves of $\mathcal{ML}_g$
are analogues of integral points in Euclidean space, and the function $B(X,q)$ measures the growth rate of 
integral multi-curves. See Proposition \ref{lem:limit} below. 
Thus the estimate of $B(X,q)$ can be reduced to length estimates of simple closed geodesics 
on the flat surface. 

\subsection{Teichm\"uller space and moduli space}
Let $S=S_g$ be an oriented closed surface of genus $g\geq 2$.
Let $\mathcal{T}_g$ be the Teichm\"uller space of marked hyperbolic metrics on $S$. Given a simple closed curve $\gamma$ on $S$, we denote by $\ell_X(\gamma)$ 
the geodesic length of $\gamma$ on $X\in \mathcal{T}_g$. 
Let $\operatorname{Mod}_g$ be the mapping class group of $S$. The group $\operatorname{Mod}_g$ acts on $\mathcal{T}_g$ and the quotient 
$\mathcal{M}_g = \mathcal{T}_g / \operatorname{Mod}_g$ is the moduli space of Riemann surfaces of genus $g$. 

Given $X\in \mathcal{T}_g$, the cotangent space of $\mathcal{T}_g$ at $X$ is naturally identified with $Q(X)$, i.e., the space of holomorphic quadratic differentials on $X$. Let $Q\mathcal{T}_g$ be the cotangent bundle and set
$Q\mathcal{M}_g =  Q\mathcal{T}_g / \operatorname{Mod}_g$.
Denote the subset of $Q\mathcal{M}_g$ consisting of quadratic differentials of unit area by  $Q^1\mathcal{M}_g$.

Let $\Pg$ be the principal stratum of $Q^1\mathcal{M}_{g}$, 
which consists of quadratic differentials with $4g-4$ simple zeros. It is known that $\Pg$ is a dense set of full measure (with
 respect to Lebesgue measure class).

\subsection{Measured laminations}
 Endow $S$ with a hyperbolic metric. A \emph{geodesic lamination} on $S$ is a closed subset of $S$ which is a disjoint union of simple geodesics. A \emph{measured lamination} $\lambda$ is a geodesic lamination (called the support of $\lambda$) that carries a transverse invariant measure. Denote the space of measured laminations on $S$ by $\ML$. 
The space $\ML$ turns out to be a  finite
dimensional manifold, homeomorphic to  $\mathbb{R}^{6g-6}$.

Weighted simple closed geodesics are dense in $\ML$. 
If $\gamma=\sum a_i \gamma_i \in \ML$, where $a_i\in \mathbb{N}$ and $\gamma_i$'s are pairwise disjoint simple closed geodesics (no two of which are homotopic), then we call $\gamma$ an \emph{integral multi-curve}. 
Let $\ML(\mathbb{Z})$ be the set of integral multi-curves in $\ML$.
It is an analogue of integral points in $\mathbb{R}^{6g-6}$. 

There is a piecewise linear structure on $\ML$ given by train track coordinates. 
Recall that a \emph{train track} is an embedded $1$-complex $\tau \subset S$
 whose edges (called \emph{branches}) are smooth arcs with well-defined tangent vectors at
 the vertices.  At any vertex (called a \emph{switch}) the incident edges are mutually
 tangent. We shall assume that every switch is trivalent and  every
 complementary region  of $\tau$ is a disc, either with three cusps at
 the boundary or with one cusp
 at the boundary. Moreover, we assume that the train track is
\emph{recurrent} in the sense that it admits a transverse measure which is positive on every branch.
A train track satisfying with above properties is called \emph{complete}. 

For a complete train track $\tau$, the set of weights on $\tau$ satisfying the switch conditions is an open convex cone of $\mathbb{R}^{6g-6}$. It is corresponding to an open set $U(\tau)\subset \ML$,
where integral weights in the train track coordinates are correspondence with integral multi-curves in $U(\tau)$. We say that every $\lambda\in U(\tau)$ is carried by $\tau$. It turns out that the Euclidean measure of $\mathbb{R}^{6g-6}$ induces a volume form $\mu_{\operatorname{Th}}$ on $U(\tau)$, which is called the \textit{Thurston measure}. 
The transition maps between different train track coordinates are linear and preserve integral points.
As a consequence, the Thurston measure is well-defined on $\ML$ and it is  mapping class group invariant.

For $\lambda\in U(\tau)$, we write $\lambda = \lambda_1\oplus_\tau \lambda_2$
if the train track coordinates of $\lambda$ is the sum of those of $\lambda_1$ and $\lambda_2$. 

Let $(X,q)$ be a holomorphic quadratic differential. Consider the length function $\ell_q$ on $\ML$.
The next lemma shows that $\ell_q$ is convex on $U(\tau)$. 

\begin{lemma}\label{lem:convex}  
For any $\lambda_1, \lambda_2\in U(\tau)$, we have 
$$\ell_q(\lambda_1\oplus_\tau \lambda_2) \leq \ell_q(\lambda_1)+ \ell_q(\lambda_2).$$
\end{lemma}
\begin{proof}
The lemma holds when $\rho$ is a Riemannian metric of negative curvature,
as shown by \cite[Theorem 5.1]{EskinAlex2022Ecos}. 
Their proof can be adapted to the flat metric induced by $q$,
because the metric can be realized as a geodesic current.
\end{proof}

It follows from Lemma \ref{lem:convex} that the unit ball $U(X,q)$ is locally convex. As a consequence, we have 

\begin{proposition}\label{lem:limit}
  For any quadratic differential $(X,q)$, we have 
  $$B(X,q)=\lim_{L\to\infty} \frac{\# \{\gamma \in \ML(\mathbb{Z}) : \ell_q(\gamma) \leq L\}}{L^{6g-6}}.$$
\end{proposition}
\begin{proof}
  This is a corollary of Lemma \ref{lem:convex}.
Compare with \cite[Proposition 3.1]{Mir2008}.
\end{proof}

\subsection{Dehn-Thurston coordinates}

Let  $\mathcal{P}=\{\alpha_1, \cdots, \alpha_{3g-3}\}$  be a pants decomposition of  $S$. 
Associated to $\mathcal{P}$ is a parametrization of $\ML(\mathbb{Z})$ called \emph{Dehn-Thurston coordinates}. Roughly speaking, the Dehn-Thurston coordinates
 for an integral multi-curve $\gamma$ consist of
 \begin{itemize}
   \item the length parameters $i(\gamma, \alpha_j)$, which is intersection  number
   of $\gamma$ with each $\alpha_j$, and,
   \item the twist parameters $\operatorname{tw}(\gamma, \alpha_j)$, which measures the twist of $\gamma$ around each $\alpha_j$.
 \end{itemize}
See \cite{PH} for details. 
For any $\gamma\in \ML(\mathbb{Z})$, we
denote its Dehn-Thurston coordinates by
\begin{eqnarray*}
 \left( n_j, t_j\right)_{j=1}^{3g-3} \in \left( \mathbb{Z}_{\geq 0} \times \mathbb{Z} \right)^{3g-3},
\end{eqnarray*}
where $n_j=i(\gamma, \alpha_j)$ is the intersection number of $\gamma$ with $\alpha_j$,
and $t_j=\operatorname{tw}(\gamma, \alpha_j)$ is the twist number of $\gamma$ around $\alpha_j$.

\section{Lower bounds and limits}\label{sec:limit}

In this section we prove Theorem \ref{thm:limit}. 
The proof is similar to that of \cite[Proposition 3.6]{Mir2008},
by relating $\ell_q(\gamma)$ (where $\gamma$ is an integral multi-curve)
to certain combinatorial length of $\gamma$. 

\subsection{An uniform lower bound for $B(X,q)$.}
Let $X$ be a Riemann surface and let $\gamma$ be a simple closed curve $\gamma$ on $X$. For any conformal metric $\rho$ on $X$, we denote by $L_\rho(\gamma)$ the infimum of the $\rho$-length of closed curves in the homotopy class of $\gamma$.
The extremal length of  $\gamma$ on $X$ is defined by 
$$\operatorname{Ext}_X(\gamma) = \sup_\rho L^2_\rho(\gamma),$$
where $\rho$ is taken over all conformal metric on $X$ of unit area.  The extremal length of simple closed curves extends continuously to measured laminations (foliations) \cite[Proposition 3]{Kerckhoff}.

By taking $\rho=|q|^{1/2}$, we have
$$\ell_q(\gamma) \leq \sqrt{\operatorname{Ext}_X(\gamma)}.$$
This gives
\begin{eqnarray*}
  B(X, q) &\geq &   \Lambda(X) = \mu_{\Th} \left(\{\lambda\in \ML :  \sqrt{\operatorname{Ext}_X(\gamma)} \leq 1\}\right).
\end{eqnarray*}
The function $\Lambda(X)$ appears in \cite{ABEM}, which is called the \textit{Hubbard-Masur function}.
Dumas \cite[Theorem 5.19]{Dumas} proved that it is a constant on $\mathcal{M}_g$.

By an inequality of Maskit \cite{Maskit}, for any simple closed curve $\gamma$
on $X$, we have $$\operatorname{Ext}_X(\gamma) \leq \frac{\ell_X(\gamma)}{2}  e^{\ell_X(\gamma)/2}.$$ Since $\ell_{q}(\gamma) \leq \sqrt{\operatorname{Ext}_{X}(\gamma)}$, we have $$\ell_{q}(\gamma) =O\left( \sqrt{\ell_{X}(\gamma)} \right).$$

\subsection{Thick-thin decomposition}

Next we describe the Deligne-Mumford compactification of quadratic differentials on Riemann surfaces. 

Fix a sufficiently small constant $\epsilon>0$. 
Let $(X,q)$ be a quadratic differential in  $ Q\mathcal{M}_g$.
Assume that there is a maximal collection of disjoint simple closed curves 
$\{\alpha_1,\ldots,\alpha_k\}$ 
such that $\ell_{X}(\alpha_j) < \epsilon$ for all $1\leq j\leq k$. Therefore, $\ell_{X}(\gamma) \geq \epsilon$ for any simple closed curve $\gamma$
 that is not contained in $\{\alpha_1, \cdots, \alpha_k\}$.

We can represent each $\alpha_j$ on $X$ as a geodesic of $q$. There are two possibilities: 
\begin{itemize}
\item[(i)]  $\alpha_j$ is contained in a maximal flat cylinder of $q$, denoted by
$C_j$. 
Then
the circumference of $C_{j}$ is equal to $\ell_{q}(\alpha_j)$. 
  Let $h_{q}(\alpha_j)$ be the height of $C_{j}$.  
  The area of $C_{j}$ is equal to 
  $ \ell_{q}(\alpha_j)h_{q}(\alpha_j).$
\item[(ii)] The other case is that there is no flat cylinder of $q$ in the homotopic class of $\alpha_j$. In this case, the geodesic representation of 
    $\alpha_j$ is unique, consisting of a finite concatenation of saddle connections. 
   By abuse of notation, 
we say that $\alpha_j$ is contained in a degenerate flat cylinder $C_{j}$ with height $h_{q}(\alpha_j)=0$.
\end{itemize}

By removing all the cylinders $C_{j}$ ($j=1,\cdots, k$) from $X$,  the complement is a union of subsurfaces $\Omega_1, \cdots, \Omega_\ell$ bounded by the geodesic representation of $\{\alpha_j\}_{j=1}^k$. We write
$$(X, q) = \left( \bigcup_{j=1}^k C_j \right) \bigcup  \left( \bigcup_{j=1}^\ell \Omega_j \right),$$
which is called the \emph{thick-thin decomposition} of $(X, q)$. 
Each $\Omega_j$ is called a \emph{thick component} of $(X, q)$.
A thick component  $\Omega_j$ is called \emph{stable} if it is  not homeomorphic to a pair of pants.

\subsection{Limits of  quadratic differentials} 
Let $(X_n, q_n)$ be a divergent sequence of quadratic differentials in $ Q^1\mathcal{M}_g$.
We assume that
$X_n$ converges to $X_\infty$ in the Deligne-Mumford compactification, 
where $X_\infty$ denotes a noded Riemann surface.
By definition, there is a maximal collection of disjoint simple closed curves 
$\{\alpha_1,\ldots,\alpha_k\}$ 
such that $\ell_{X_n}(\alpha_j) \to 0$ for all $1\leq j\leq k$. So for $n$ sufficiently large we have $\ell_{X}(\alpha_j) < \epsilon$ for all $1\leq j\leq k$,
and $\ell_{X_n}(\gamma) > \epsilon$ for any simple closed curve $\gamma$
 that is not contained in $\{\alpha_1, \cdots, \alpha_k\}$.

Up to a subsequence, we may assume that $q_n$ converges to  
a quadratic differential $q_\infty$ on $X_\infty$, with at most simple poles at the punctures. The convergence can be described as follows. 
Let $U$ be any sufficiently small neighborhood of the punctures of $X_\infty$. 
Then for $n$ sufficiently large there is a conformal embedding $F_n : X_\infty \setminus U \to X_n$ such that $F_n^* q_n \to q_\infty$ as $n\to \infty$ uniformly on $X_\infty \setminus U$. Note that the limited quadratic differential $q_\infty$ may be vanishing on some component of $X_\infty$. In particular, $q_\infty \equiv 0$
on every component of $X_\infty$ that is $3$-punctured sphere, since there is no non-zero meromorphic   quadratic differential of finite area on a  $3$-punctured sphere.

In the following discussion, we will assume that the thick-thin decomposition of $(X_n, q_n)$ is given by
$$(X_n, q_n) = \left( \bigcup_{j=1}^k C_j \right) \bigcup  \left( \bigcup_{j=1}^\ell \Omega_j \right),$$
 In abbreviated notation, here we denote by $C_j$ and $\Omega_j$ be the cylinder components and thick components of $(X_n, q_n)$, without referring to the explicit quadratic differential $q_n$. We shall write $m(C_j, q_n)$ or $m(\Omega_j, q_n)$
 to be the $|q_n|$-area of a component.

\begin{lemma}\label{lem:stable}
Let $\Omega$ be a stable component in the thick-thin decomposition of $(X_n, q_n)$. 
Denote the restriction of $q_n$ on $\Omega$ by $(\Omega, q_n)$.
If $m(\Omega, q_n)$ is bounded away from $0$,
then by passing to a subsequence, 
$(\Omega, q_n)$ converges to a nonzero finite area quadratic differential
$(\Omega, \psi)$.
\end{lemma}
\begin{proof}
 The restriction of each $q_n$ on $\Omega$ defines a Riemann surface structure on $\Omega$, denoted by $Y_n$.
Then the hyperbolic length of every essential closed curve on $Y_n$ (neither homotopic to a point nor  to a boundary component) is bounded below away from $0$. If $\alpha_j$ is a boundary component of $Y_n$,
then $\ell_{Y_n}(\alpha_j)$ can tend to $0$ or not. In either case, it is clear that 
there is no flat cylinder of $q_n$ in the homotopy class of $\alpha_j$. 

By Lenzhen-Masur \cite[Theorem 2]{LM2010}, if every boundary component of 
$(\Omega, q_n)$ is not homotopic to a flat cylinder 
and $m(\Omega, q_n)$ is bounded away from $0$,
then by passing to a subsequence $(\Omega, q_n)$ converges to a nonzero finite area quadratic differential $(Y_\infty, \psi)$. The convergence is in the sense that for any neighborhood $U$ of the possible punctures on $Y_\infty$,
for $n$ sufficiently large there exists a conformal mapping $F_n' : Y_\infty \setminus U \to Y_n$
such that $(F_n')^*q_n$ converges to $\psi$ uniformly.
\end{proof}

\begin{remark}
In general, we can scale $(\Omega, q_n)$ by the factor $1/ m(\Omega, q_n)$ to obtain a sequence of quadratic differentials with unit area on the subsurface. 
Then the proof of \cite[Theorem 2]{LM2010}
shows that the sequence of quadratic differentials (after scaling) converges to some $(\Omega, \psi)$ of unit area.
\end{remark}

\subsection{Length estimates}

In the next lemma, $C_j$ denotes a cylinder component of $(X_n,q_n)$.
\begin{lemma}\label{lem:spiral}
If $\beta$ is a saddle connection that intersects with $C_j$, then it spirals around $C_j$ at most once 
each time before entering into $C_j$.
\end{lemma}

\begin{proof}
Let $A_j$  be the annular covering  of $ X_n$
such that $C_j$ naturally embeds into  $A_j$.
Let $\tilde{\beta}$ be a lift of $\beta$, which is an arc crossing through 
 $C_j$.
 
Consider a segment $\beta_1$ of  $\tilde{\beta}$,
which starts at one of the endpoints of $\tilde{\beta}$  and stops until it
hits $C_j$ for the first time.  We claim that $\beta_1$ spirals around the boundary of $C_j$ more than once.

 In fact, each boundary component of $C_j$ is a geodesic representation of $\alpha_j$ and contains at least one critical point. 
Near a critical point the flat metric is obtained by gluing  at least 
 three rectangle along their horizontal edges. If $\beta_1$ spirals around the boundary of $C_j$ more than once, then there must be a horizontal critical trajectory $\alpha$ of $q_n$ that lies in $A_j\setminus C_j$ 
 and intersects with $\beta_1$ at least twice. 
As a result,  we can make $\beta_1$ even shorter, which is contradicted with the assumption that it is a segment of a saddle connection. 

\end{proof}

\begin{figure}[htb]
\begin{tikzpicture}[scale=0.8]
\begin{scope}[very thick,dashed]
\draw[thin] (0,0) circle (1cm);
\draw (0,0) circle (2cm);
\draw (-2,0) -- (-4,0);
\draw (2,0) -- (1,0);
\draw[thin] (0,0) circle (4cm);
\end{scope}
\draw[red,thick] (-4,0) arc (180 : 0 : 3.5);
\draw[red,thick] (3,0) arc (0 : -180 : 2.75);
\draw[red,thick] (-2.5,0) arc (180 : 90 : 2);
\draw[fill=gray] (0,0) circle (1cm) circle (2cm);
\draw[fill=white] (0,0) circle (1cm);
\node[left] at (-4,0) {$x$};
\node[left] at (-0.3,2.2) {$y$};
\node[left] at (-0.3,2.2) {$y$};
\node[below] at (-2.65,0) {$z$};
\node[right] at (-2,0) {$p$};
\node[right] at (3,0) {$\beta_1$};
\end{tikzpicture}
	\caption{In the figure, $\beta_1$ is represented by an arc connecting $x$ to $y$. If $\beta_1$ spirals around $C_j$ more than once, then it must intersect a horizontal leaf of $q_n$ at a point $z$. The union of the horizontal segments
$\widehat{zp}$ and $\widehat{py}$ is shorter than the segment of $\beta_1$ from $z$ to $y$.} 
	\label{fig_04}
\end{figure}
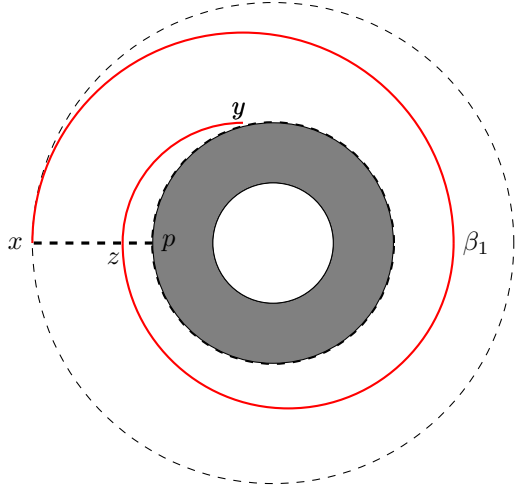

\begin{proof}[Proof of Theorem \ref{thm:limit}]
Up to a subsequence, we 
assume that $\{\alpha_1,\cdots,\alpha_k\}$ is a maximal collection of 
simple closed curves that are pinched along the sequence of Riemann surfaces $X_n$. 
We expand $\{\alpha_1,\cdots,\alpha_k\}$ into a pants decomposition $$\mathcal{P}=\{\alpha_1,\cdots,\alpha_k, \alpha_{k+1}, \cdots, \alpha_{3g-3}\}.$$
Again, up to a subsequence, we may assume that there is a constant $C>0$ such that 
$\ell_{q_n}(\alpha_j) \leq C$ for all $1\leq j \leq 3g-3$. 

Given a simple closed curve $\gamma$, we denote by $\left(n_j, t_j \right)_{j=1}^{3g-3}$
the Dehn-Thurston coordinates of $\gamma$ associated to $\mathcal{P}$. 
To simplify our notation, we assume that $\gamma$ is a geodesic of $q_n$. We decompose $\gamma$ into a finite union of adjacent straight segments $s_i$, 
such that each $s_i$ is either contained in some thick component $\Omega_j$ or contained in some cylinder $C_j$.

Denote the set of segments which are contained in 
the thick components by $\mathcal{S}_0$.  The set of segments that are contained in $\bigcup\limits_{j=1}^k C_j$ is denoted by $\mathcal{S}_1$.

For any $s_i\in \mathcal{S}_1$, either it crosses some $C_j$ or it is a saddle connection contained in the boundary of some $C_j$. 
Denote by $t_j(s_i)$ the twisting number of $s_i$ around $\alpha_j$. 

By Lemma \ref{lem:spiral}, for each $C_j$, we have  

$$\sum_{\substack{s_i\in \mathcal{S}_1 \\
s_i\subset C_j}} |t_j(s_i)|\leq |t_j|+2n_j.$$

If  $C_j$ is a flat cylinder such that $h_j(q_n)>0$, then 
we rotate such that the closed leaves in $C_j$ are vertical.
We are able to bound the length of $s_i\subset C_j$
in terms of its horizontal and vertical length. 
Note that the horizontal length of $s_i$ is less than $h_{q_n}(\alpha_j)$ and the vertical length of $s_i$ is less than $\left( 1+t_j(s_i) \right) \ell_{q_n}(\alpha_j)$. 
If $C_j$ is a degenerate cylinder, then the length of $s_i$ is  bounded above by
$\left( 1+t_j(s_i) \right) \ell_{q_n}(\alpha_j)$.
As a result, we have

\begin{eqnarray}\label{equ:thin1}
 \sum_{\mathcal{S}_1} \ell_{q_n}(s_i) & \leq &  \sum_{j=1}^k  \left(|t_j|+2n_j \right) \ell_{q_n}(\alpha_j)
  + n_j h_{q_n}(\alpha_j) \nonumber \\
  &=&  \sum_{j=1}^k  \left(h_{q_n}(\alpha_j) +2 \ell_{q_n}(\alpha_j) \right) n_j + \ell_{q_n}(\alpha_j) |t_j|.
\end{eqnarray}

Now assume that $s_i\in \mathcal{S}_0$ and $s_i$ is contained in $\Omega_j$. 
First assume that $\Omega_j$ is a pair of pants.
Using Lemma \ref{lem:spiral} again,  $s_i$ spirals around each boundary component of
$\Omega_j$ at most once. As shown by Rafi \cite[Theorem 4]{Rafi07}, the diameter of  $\Omega_j$
is comparable to its size, which is given by the maximal length of the boundary components. It follows that each $s_i$ in $\Omega_j$ contributes a small length and the sum over all possible pairs of pants $\Omega_j$ 
$$\sum_{\Omega_j }\sum_{\substack{s_i\in \mathcal{S}_0 \\
s_i\subset \Omega_j}} \ell_{q_n}(s_i)$$
can be absorbed by the term \eqref{equ:thin1}.

It remains to consider those $s_i$ that are contained in the stable thick components. For each stable component $\Omega_j$, we assume that,  up to scaling by
$m(\Omega_j,q_n)$, the sequence $(\Omega_j, q_n)$ converges to a non-zero quadratic differential $(\Omega_j,\psi_j)$. Denote 
$$\psi=\bigcup_j (\Omega_j,\psi_j). $$
For any $s_i \in \mathcal{S}_0$ that is contained in the interior of  $\Omega_j$, 
when $n$ is sufficiently large we have 
$$\ell_{q_n}(s_i) \asymp \sqrt{m(\Omega_j,q_n)} \ell_{\psi_j}(s_i),$$
where $\psi_j$ is a non-zero quadratic differential on a connected component of $X_\infty$ corresponding to $\Omega_j$. 
So we have
\begin{equation}\label{equ:thick}
 \sum_{\Omega_j} \sum_{\substack{s_i\in \mathcal{S}_0 \\
s_i\subset \Omega_j}}  \ell_{q_n}\leq  \sum_{\Omega_j} \sqrt{m(\Omega_j,q_n)} C(\psi_j)   \sum_{\substack{k+1\leq j \leq 3g-3 \\ \alpha_j \subset \Omega_j}}   (|t_j| + n_j),
\end{equation}
where the sum is taken over all stable components $\Omega_j$,
and each $C(\psi_j)$ is a positive constant depending only on the quadratic differential $\psi_j$.

Combining \eqref{equ:thin1} with \eqref{equ:thick}, 
we conclude that for $n$ sufficiently large the length $\ell_{q_n}(\gamma)$ is bounded above by 
\begin{eqnarray}\label{equ:upperbound}
\ell_{q_n}(\gamma) &\leq&  \sum_{j=1}^k  \left(h_{q_n}(\alpha_j) + 2 \ell_{q_n}(\alpha_j) \right) n_j + \ell_{q_n}(\alpha_j) |t_j| \nonumber \\
&+& C  \sum_{\Omega_j} \sqrt{m(\Omega_j,q_n)}  \sum_{\substack{k+1\leq j \leq 3g-3 \\ \alpha_j \subset \Omega_j}}   (|t_j| + n_j),
\end{eqnarray}
where $C=C(\psi)$. It is easy to extend the above inequality to any multi-curve on the surface. 
With this inequality, applying Proposition
\ref{lem:limit} and the same argument of Mirzakhahi \cite[Proposition 3.6]{Mir2008},
we are able to find a lower bound for $B(X_n, q_n)$:
$$
B(X_n,q_n) \geq \frac{c}{\prod\limits_{\Omega_j} m(\Omega_j, q_n) \prod\limits_{j=1}^k \left( m(C_j, q_n) + \ell_{q_n}(\alpha_j)^2 \right)},
$$
where the product $\prod\limits_{\Omega_j}$ is taken over all stable components,
and here the constant $c>0$ only depends on  $\psi$.

If there is a stable component $\Omega_j$ such that $m(\Omega_j, q_n)\to 0$,
then $B(X_n, q_n) \to \infty$.  
And note that $\ell_{q_n}(\alpha_j) \to 0$ for all $j=1, \cdots, k$.  
So if there is some cylinder $C_j$ such that $m(C_j, q_n)=h_{q_n}(\alpha_j) \ell_{q_n}(\alpha_j) \to 0$,
then we also have $B(X_n, q_n) \to \infty$.  
This confirms Theorem \ref{thm:limit} in case that the sequence $(X_n,q_n)$
does not contain any regular subsequence. 

The remaining case is that there is a positive constant $\delta>0$ such that for each stable component $\Omega_j$
and each $C_j$,  $m(\Omega_j, q_n)$ and $m(C_j, q_n)$ are bounded below by $\delta$. 
In particular, every $\alpha_j$, $j=1, \cdots, k$, is contained in a flat cylinder. 
In this case, by modifying the proof of \eqref{equ:thin1}, we obtain
\begin{eqnarray}\label{equ:thin2}
 \sum_{\mathcal{S}_1} \ell_{q_n}(s_i) & \geq &  \sum_{j=1}^k  \left(|t_j|-2n_j \right) \ell_{q_n}(\alpha_j)
  + n_j h_{q_n}(\alpha_j) \nonumber \\
  &=&  \sum_{j=1}^k  \left(h_{q_n}(\alpha_j) -2 \ell_{q_n}(\alpha_j) \right) n_j + \ell_{q_n}(\alpha_j) |t_j|.
\end{eqnarray} 
Since $\ell_{q_n}(\alpha_j) \to 0$ and $h_{q_n}(\alpha_j) \ell_{q_n}(\alpha_j) >\delta$, we have 
$h_{q_n}(\alpha_j) \to \infty$. 
Thus the lower bound \eqref{equ:thin2} is approximately 
$$\sum_{j=1}^k  h_{q_n}(\alpha_j)  n_j + \ell_{q_n}(\alpha_j) |t_j|.$$
Similarly, for those $s_i$ that are contained in the stable components,
we have 
\begin{equation}\label{equ:thick2}
  \sum_{\mathcal{S}_0} \ell_{q_n}(s_i) \geq   \sum_{\Omega_j} c(\psi_j) \sqrt{m(\Omega_j,q_n)}   \sum_{\substack{k+1\leq j \leq 3g-3 \\ \alpha_j \subset \Omega_j}}   (|t_j| + n_j),
\end{equation}
where $c(\psi_j)$ is a positive constant depending only on the quadratic differential $\psi_j$.
Thus \eqref{equ:thin2} and \eqref{equ:thick2} together imply
\begin{eqnarray*}
B(X_n,q_n) &\leq& C \prod_{\Omega} \frac{1 }{m(\Omega_j, q_n)} \prod_{j=1}^k \frac{1}{m(C_j, q_n)}  \\
&\leq& C \prod_{\Omega_j} \frac{1 }{\delta} \prod_{j=1}^k \frac{1}{\delta},
\end{eqnarray*} 
where $C=C(\psi)$.
Thus the proof of the theorem is complete.

\end{proof}

\begin{proof}[Proof of Corollary \ref{thm:unit}]
We must have $m(C_j, q_n) \to 0$ for all $j=1, \cdots, k$. 
For otherwise,  $m(X_\infty, q_\infty)<1$, which is contradicted to our assumption. 
The lower bound of $B(X_n,q_n)$ in the proof of Theorem \ref{thm:limit} shows that 
$B(X_n,q_n) \to \infty$.
  
\end{proof}

\section{Upper bounds and integration}
\label{sec:collar}
In this section we prove Theorem \ref{thm:int}. 
As we have noted before,
there is no standard collar lemma for quadratic differentials. 
Thus length estimate for quadratic differentials is more difficult. 

\subsection{Collar lemma for quadratic differentials}

Let $(X,q)$ be a quadratic differential with unit area. 
We endow $X$ with the metric $|q|^{1/2}$.

\begin{definition}
    Let $A$ be an annulus on $X$ such that there are no zeroes of $q$ in the interior.
   We say that $A$ is \emph{expanding} if the curvature of each boundary component of $A$ has constant sign, either positive or
negative at each point, and the boundary curves are equidistant.
\end{definition}
We denote the boundary components of an expanding annulus $A$ by $\partial_0 A$ and $\partial_1 A$. Without loss of generality, we  assume that the total curvature $\kappa(\partial_0 A)= \kappa(A)> 0$ and $\kappa(\partial_1 A)= -\kappa(A) < 0$. 
Let $d(A)$ be the distance between $\partial_0 A$ and $\partial_1 A$,
which is called the \emph{width} of the expanding annulus.

In the following, we assume that $\gamma$ is a simple closed curve on $X$. Let $C=C(\gamma)$ be the maximal flat cylinder on
$q$ with core curves homotopic to $\gamma$, and let $A=A(\gamma)$ be the maximal expanding annulus,  one of whose boundary components is a $q$-geodesic homotopic to $\gamma$.  
Denote the conformal modulus of $A$ and $C$ by $\operatorname{Mod}(A)$
and $\operatorname{Mod}(C)$, respectively. Note that $C$ is degenerate when the
 geodesic representative of $\gamma$ is unique (in this case $\operatorname{Mod}(C)=0$).  
 
 The collar lemma for quadratic differentials is established by Minsky
 \cite{Minsky92} and Choi-Rafi-Series \cite{CRS08}. The following statement can be found in Lenzhen-Masur \cite[Lemma 1]{LM2010}.

\begin{lemma}\label{lem:annulus}
  Assume that $q\in Q^1(X)$ and $\gamma$ is a simple closed curve such that $\ell_X(\gamma)$ is sufficiently small. Then  
  \begin{enumerate}
      \item[(1)]  $$\frac{1}{\ell_X(\gamma)} \asymp \max\left\{ \operatorname{Mod}(A), \operatorname{Mod}(C)\right\};$$
      \item[(2)] $$\operatorname{Mod}(A) \asymp \log \frac{d(A)}{\ell_q(\gamma)}.$$
  \end{enumerate}
\end{lemma}

By Lemma \ref{lem:annulus},
either $\frac{1}{\ell_X(\gamma)} \asymp \operatorname{Mod}(C)$
or $\frac{1}{\ell_X(\gamma)} \asymp \operatorname{Mod}(A)$. 
In the first case, $C$ is a flat cylinder, so we have $$\operatorname{Mod}(C)= \frac{h(C)}{\ell_q(\gamma)},$$ where $h(C)$ denotes the width (or height) of $C$. 
Thus we have 
\begin{equation*}\label{equ:flat}
h(C) \asymp \frac{\ell_q(\gamma)}{\ell_X(\gamma)}.
\end{equation*}

In the second case, we have 
\begin{equation*}\label{equ:expanding}
    \log \frac{d(A)}{\ell_q(\gamma)} \asymp \operatorname{Mod}(A)\asymp \frac{1}{\ell_X(\gamma)},
\end{equation*}
which implies 
$$d(A) \succ \ell_q(\gamma) e^{1/\ell_X(\gamma)} \geq \frac{\ell_q(\gamma)}{\ell_X(\gamma)}.$$

In conclusion, if $\ell_X(\gamma)$ is sufficiently small, then there is an annulus homotopic to $\gamma$ 
whose width is at least 
\begin{equation}\label{equ:width}
c \frac{\ell_q(\gamma)}{\ell_X(\gamma)},
\end{equation}
where $c>0$ is a topological constant. 

\begin{remark}
There is a related result of Rafi \cite[Theorem 2]{Rafi07}, which states that if $\ell_X(\gamma)< L$ and $\alpha$
is a simple closed curve that intersects with $\gamma$, then there is a constant $D=D(L)$ such that 
\begin{equation}\label{equ:rafi}
\ell_q(\alpha) \geq \frac{\ell_q(\gamma)}{D}.
\end{equation}
The lower bound \eqref{equ:width} is better than \eqref{equ:rafi} when $\ell_X(\gamma)$ is small. Note that for $q\in Q^1(X)$, we have $\ell_X(\gamma)\to 0$ if $\ell_q(\gamma) \to 0$. However, there is no quantitative estimate of   $\ell_X(\gamma)$, and, unlike hyperbolic metrics,  small $\ell_q(\gamma)$ does not imply that the width of the cylinder is large. 
\end{remark}

\subsection{Upper bounds}

Let $\epsilon>0$ be a fixed small constant. It is obvious that if $X$ lies in the $\epsilon$-thick part of $\mathcal{M}_g$, then there is a constant $C(\epsilon)$ such that $B(X,q)<C(\epsilon)$ for any $q\in Q^1(X)$. 

In the following, we assume that $X$ is in the $\epsilon$-thin part of $\mathcal{M}_g$. We can cover  the $\epsilon$-thin part of $\mathcal{M}_g$ by finitely many open sets $\mathcal{U}_\epsilon^i$ such that for $X\in \mathcal{U}_\epsilon^i$ we have:
\begin{itemize}
  \item $\{\gamma_1, \cdots, \gamma_k\}$ is the maximal collection of short curves on $X$, i.e. $\ell_X(\gamma_j)<\epsilon$.
  \item We can expand $\{\gamma_1, \cdots, \gamma_k\}$ to a pants decomposition $\mathcal{P}=\{\gamma_1, \cdots, \gamma_{3g-3}\}$ such that $\ell_X(\gamma_j)\leq B_g$, where $B_g$ is Bers constant.
\end{itemize}  

\begin{theorem}\label{thm:upper}
    Let $X\in \mathcal{U}_\epsilon^i$. Then for any $q\in Q^1(X)$, we have 
    $$B(X,q)\leq C \prod_{j=1}^{3g-3} \frac{\ell_X(\gamma_j)}{\ell_q(\gamma_j)^2}. $$
\end{theorem}

\begin{proof}

    Given any simple closed curve $\gamma$, let 
    $$\left(n_1, n_2, \cdots, n_{3g-3},t_1, t_2, \cdots, t_{3g-3}\right)$$
    be the Dehn-Thurston coordinates of $\gamma$ associated to $\mathcal{P}$. By Lemma \ref{lem:annulus}, for each $\gamma_j$, there is an annulus $A(\gamma_j)$ of width $d_j$ satisfying 
    $$d_j \geq c \frac{\ell_q(\gamma_j)}{\ell_X(\gamma_j)}.$$

    Since $\gamma$ goes through $A(\gamma_j)$ about $n_j$ times and all $A(\gamma_j)$ are pairwise disjoint, we have
\begin{eqnarray*}
\ell_q(\gamma) &\geq &  \frac{\sqrt{2}}{2}\sum \limits_{j=1}^{3g-3} \left( n_jd_j+|t_j| \ell_q(\gamma_j) \right) \\
&\geq&   \frac{\sqrt{2}}{2} \sum \limits_{j=1}^{3g-3} \left( c \frac{\ell_q(\gamma_j)}{\ell_X(\gamma_j)}n_j +|t_j| \ell_q(\gamma_j) \right). \\
\end{eqnarray*}

    By our previous estimate of $B(X,q)$, we have
    $$B(X,q)\leq C \prod_{j=1}^{3g-3} \frac{\ell_X(\gamma_j)}{\ell_q(\gamma_j)^2}. $$   

\end{proof}

Note that  $\ell_X (\gamma_j) \leq \pi \operatorname{Ext}_X(\gamma_j)$ (see \cite{Maskit}), so we have
\begin{corollary} There is a constant $C$ such that for any $X\in \mathcal{U}_\epsilon^i$ and $q\in Q^1(X)$, we have 
    \begin{equation}\label{equ:extremalbound}
B(X,q)\leq C\prod_{j=1}^{3g-3}\frac{\operatorname{Ext}_X(\gamma_j)}
{\ell^2_q(\gamma_j)}.
\end{equation}
\end{corollary}

\subsection{$p$-integrability}
Now we use Theorem \ref{thm:upper} to prove Theorem \ref{thm:int}.
\begin{proof}[Proof of Theorem \ref{thm:int}]
It suffices to restrict  $B(X,q)$ on  the principal stratum  $\Pg$.

Fix $0<p<1$. Without loss of generality,
we may assume that $(X,q)$ lies in the thin part of the moduli space. 
Since the $\epsilon$-thin part of $\mathcal{M}_g$ can be covered by finitely many open sets $\mathcal{U}_\epsilon^i$, it suffices to assume that $X$ lies in some fixed $\mathcal{U}_\epsilon^i$.
As in the proof of Theorem \ref{thm:upper}, by
choosing an appropriate pants decomposition $\{\gamma_1, \cdots, \gamma_{3g-3}\}$, we have
\begin{eqnarray*}
B(X,q)^p &\leq&  C \prod^{3g-3}_{i=1} \left( \frac{\ell_X(\gamma_j)}{\ell^2_q(\gamma_j)}\right)^p \\
&=& O\left( \prod^{3g-3}_{j=1} \frac{1}{\ell^{2p}_q(\gamma_j)}  \right).
\end{eqnarray*}

Note that for each $\gamma_j$ there is a 
saddle connection $\beta_j$ contained in the geodesic representation of $\gamma_j$,
and the collection $\beta_1, \cdots, \beta_{3g-3}$ can be chosen to be independent on the periodic coordinates. So the integration of $B(X,q)^p$ with respect to the Masur-Veech measure is bounded by 
\begin{eqnarray*}
 O\left( \prod_{j=1}^{3g-3} \int_{0}^{1}\int_{0}^{1} \frac{1}{(x_j^2 + y_j^2)^p} \  dx_j dy_j  \right).
\end{eqnarray*}
The assumption $0<p<1$ ensures that the integral is finite. 
\end{proof}

\begin{remark}
   Although we know that $\ell_X(\gamma_j) \to 0$ when $\ell_q(\gamma_j)\to 0$,
  there is no efficient estimate for $\ell_X(\gamma_j)$. For instance, 
  the width of the expanding annulus of $\gamma_j$ may be small. Because of such a defect, we are not able to show that $B(X,q)$ is integrable.
\end{remark}

\subsection{Jenkins-Strebel differentials with long cylinders}

Let  $(X,q)$ be a quadratic differential. Denote by $\Gamma(q)$ the critical graph of the horizontal foliation of 
$q$. We say that $(X,q)$ is a Jenkins-Strebel differential if $\Gamma(q)$ is compact and the complement $X\setminus \Gamma(q)$  is a union of disjoint flat cylinders, each of which is foliated by horizontal closed leaves homotopic to a simple closed curve. We shall denote the horizontal foliation of $q$ by
$\sum\limits_{j=1}^{k} h_j \gamma_j,$
where $\gamma_j$ is the core curve of the cylinder $C_j$.  Denote the circumference of $C_j$ by $\ell_j$. 
Note that the area of $q$ is equal to 
$\sum\limits_{j=1}^{k} h_j \ell_j.$

For simplicity, we assume that $(X,q)$ is a maximal ($k=3g-3$) Jenkins-Strebel differential of unit area. 
Moreover, we assume that $h_j \gg 1 \gg \ell_j >0$. 
Let $\mathcal{P}=\{\gamma_j\}_{j=1}^{3g-3}$. 
Choose $\{\beta_j\}_{j=1}^{3g-3}$ such that $\{(\gamma_j, \beta_j)\}_{j=1}^{3g-3}$ defines a curve system.

It is easy to show that 
$$ h_j \leq \ell_q(\beta_j) \leq  h_j+ 1.$$
By the proof of Theorem \ref{thm:limit}, we have 
$$B(q) \geq  \prod_{j=1}^{3g-3} \frac{c}{(h_j+1)\ell_j} \asymp \prod_{j=1}^{3g-3} \frac{1}{h_j\ell_j}= 
\prod_{j=1}^{3g-3} \frac{1}{m(C_j, q)}.$$

On the other hand, note that each cylinder $C_j$ has height $h_j$, which is assumed to be large.
This is similar to the case of hyperbolic surfaces that for each  simple closed geodesic with small length $\ell$ there is a collar neighborhood with large width $\log 1/|\ell|$. 
Similar to the proof of \cite[Theorem 1.5]{AA}, one can show that 
$$B(q) \leq \prod_{j=1}^{3g-3} \frac{C}{h_j\ell_j}.$$

Note that in this special case, the upper bound of $B(X,q)$
in \eqref{equ:extremalbound}  satisfies
$$\prod_{j=1}^{3g-3}\frac{\operatorname{Ext}_X(\gamma_j)}
{\ell_j^2} \asymp \prod_{j=1}^{3g-3} \frac{\frac{\ell_j}{h_j}}
{\ell_j^2} = \prod_{j=1}^{3g-3} \frac{1}{h_j\ell_j}.$$
Thus the product in \eqref{equ:extremalbound} also gives the lower bound. 

The above estimates show that $B(q)$ is not a proper function in the moduli space. 
For instance, we may take a sequence of $q_n$ with the above combinatorial structure such that 
for all $\gamma_j$ we have $\ell_j \to 0$ while $h_j\ell_j = \frac{1}{3g-3}$. 
As $n\to \infty$ we have $B(q_n)$ bounded by a constant.

\bibliographystyle{plain}

\bibliography{bibliography}

\end{document}